\documentclass{amsart}
\usepackage{amssymb}
\newtheorem{theorem}{Theorem}[section]

\theoremstyle{definition}
\newtheorem{definition}[theorem]{Definition}
\newtheorem{prp}[equation]{Proposition}
\newtheorem{cor}[equation]{Corollary}
\newtheorem{clm}[equation]{Claim}

\newtheorem{conj}[equation]{Conjecture}

\theoremstyle{remark}
\newtheorem{remark}[theorem]{Remark}

\numberwithin{equation}{section}

\begin{document}

\title[Ergodicity for Falling Ball Systems]{Conditional Proof of the Ergodic Conjecture for Falling Ball Systems}


\author{Nandor Simanyi}
\address{1402 10th Avenue South \\
Birmingham AL 35294-1241}
\email{simanyi@uab.edu}



\date{\today}

\begin{abstract}

  In this paper we present a conditional proof of Wojtkowski's Ergodicity Conjecture
  for the system of $1D$ perfectly elastic balls falling down in a half line under constant gravitational
  acceleration, \cite{W85}, \cite{W86}, \cite{W90a}, \cite{W90b}, \cite{W98}. Namely, we prove that almost
  every such system is (completely hyperbolic and) ergodic, by assuming the transversality between different
  singularities and between singularities and stable (unstable) invariant manifolds.
  
\end{abstract}

\maketitle

\section{Introduction/Prerequisites}

In his paper \cite{W90a} M. Wojtkowski introduced the following
Hamiltonian dynamical system with discontinuities: There is a vertical
half line $\left\{q|\, q\ge 0\right\}$ given and $n$ ($\ge 2$) point
particles with masses $m_1\ge m_2\ge \dots\ge m_n>0$ and positions
$0\le q_1\le q_2\le\dots\le q_n$ are moving on this half line so that
they are subjected to a constant gravitational acceleration $a=-1$
(they fall down), they collide elastically with each other, and the
first (lowest) particle also collides elastically with the hard floor
$q=0$. We fix the total energy

\begin{equation*}
H=\sum_{i=1}^n \left(m_iq_i+\frac{1}{2m_i} p_i^2 \right)
\end{equation*}

by taking $H=1$. The arising Hamiltonian flow with collisions
$\left(\mathbf{M}, \{\psi^t \}, \mu\right)$ ($\mu$ is the Liouville measure) is the 
studied model of this paper. 

Before formulating the result of this article, however, it is worth mentioning
here three important facts:

\begin{enumerate}

\item[$(1)$] Since the phase space $\mathbf{M}$ is compact, the Liouville measure $\mu$
is finite.

\item[$(2)$] The phase points $x\in\mathbf{M}$ for which the trajectory $\{\psi^t(x)|, t \in\mathbb{R}\}$ hits at least
one singularity (i. e. a multiple collision) are contained in a countable 
union of proper, smooth submanifolds of $\mathbf{M}$ and, therefore, such points
form a set of $\mu$ measure zero.

\item[$(3)$] For $\mu$-almost every phase point $x\in\mathbf{M}$ the collision times of
  the trajectory $\{\psi^t(x)|, t \in\mathbb{R}\}$ do not have any finite accumulation point, see
  Proposition A.1 of \cite{S96}.

\end{enumerate}

In the paper \cite{W90a} Wojtkowski formulated his main conjecture pertaining
to the dynamical system $\left(\mathbf{M}, \{\psi^t \}, \mu\right)$:

\begin{conj}[Wojtkowski's Conjecture]
  If $m_1\ge m_2\ge\dots\ge m_n>0$ and
  $m_1\ne m_n$, then all but one characteristic (Lyapunov) exponents of the
  flow $\left(\mathbf{M}, \{\psi^t \}, \mu\right)$ are nonzero. Futhermore,
  the system is ergodic.
\end{conj}

\begin{remark}

1. The only exceptional exponent zero must correspond to the flow direction.

\medskip

2. The condition of nonincreasing masses (as above) is essential for 
establishing the invariance of the symplectic cone field --- an important
condition for obtaining nonzero characteristic exponents. As Wojtkowski
pointed out in Proposition 4 of \cite{W90a}, if $n=2$ and $m_1<m_2$, then there
exists a linearly stable periodic orbit, thus dimming the chances of proving
ergodicity.

\end{remark}

Our paper will very closely follow all the notations and definitions
of the atricle \cite{W90a}, so the reader is warmly recommended to be
familiar with that paper. 

\section{Eventually Strict Cone Invariance}

We will be working with the symplectic coordinates $(\delta h, \delta
v)$ for the tangent vectors of the reduced phase space $\mathbf{M}$
satisfying the usual reduction equations $\sum_{i=1}^n \delta
h_i=0=\sum_{i=1}^n \delta v_i$.

\begin{remark}
  The coordinates $\delta h_i$ and $\delta v_i$ ($i=1,2,\dots,n$)
  serve as suitable symplectic coordinates in the codimension-one
  subspace $\mathcal{T}_x$ of the full tangent space
  $\mathcal{T}_x\mathbf{M}$ of $\mathbf{M}$ at $x$. Recall that the
  $(2n-2)$-dimensional vector space $\mathcal{T}_x$ is transversal to
  the flow direction, and the restriction of the canonical symplectic
  form

  \begin{equation*}
  \omega=\sum_{i=1}^n \delta q_i \wedge \delta p_i=\sum_{i=1}^n \delta
  h_i \wedge \delta v_i
  \]
  of $\mathbf{M}$ is non-degenerate on $\mathcal{T}_x$, see \cite{W90a}. We also recall that
  \[
  \delta h_i=m_i\delta q_i+m_iv_i\delta v_i=m_i\delta q_i+v_i\delta p_i.
  \end{equation*}

\end{remark}

Corresponding to the above choice of symplectic coordinates, the
considered monotone Q-form will be

\begin{equation}
Q_1(\delta h, \delta v)=\sum_{i=1}^n \delta h_i\delta v_i.
\end{equation}

It is clear that the evolution of $DS^t(\delta h(0), \delta
v(0))=(\delta h(t), \delta v(t))$ between colisions is

\begin{equation}\label{evolution}
\frac{d}{dt}\left(\delta h(t), \delta v(t)\right)=(0, 0).
\end{equation}

If a collision of type $(i, i+1)$ ($i=1,2,\dots,n-1$) takes place at
time $t_k$, then the derivative of the flow at the collision $\delta
h^-(t_k)\mapsto \delta h^+(t_k)$, $\delta v^-(t_k)\mapsto \delta
v^+(t_k)$ is given by the matrices

\begin{equation}\label{nonfloorcollision}
\begin{aligned}
  & \delta h^+(t_k)=R^*_i\left[\delta h^-(t_k)+S_i\delta v^-(t_k)\right] \\
  & \delta v^+(t_k)=R_i\delta v^-(t_k),
\end{aligned}
\end{equation}

where the matrix $R_i$ is the $n\times n$ identity matrix, except that
its $2\times 2$ submatrix at the crossings of the $i$-th and
$(i+1)$-st rows and columns is

\begin{equation*}
  R_i^{(i,i+1)} = \left[
    \begin{array}{cc}
      \gamma_i & 1-\gamma_i \\
      1+\gamma_i & -\gamma_i
    \end{array}
\right]
\end{equation*}

with $\gamma_i=\dfrac{m_i-m_{i+1}}{m_i+m_{i+1}}$. The matrix $S_i$ is,
similarly, the $n\times n$ zero matrix, except its $2\times 2$
submatrix at the crossings of the $i$-th and $(i+1)$-st rows and
columns, which takes the form of

\begin{equation*}
S_i^{(i, i+1)}=
\left[
    \begin{array}{cc}
      \alpha_i & -\alpha_i \\
      -\alpha_i & \alpha_i
    \end{array}
\right]
\end{equation*}

with

\begin{equation}\label{alpha}
  \alpha_i=\frac{2m_im_{i+1}(m_i-m_{i+1})}{(m_i+m_{i+1})^2}(v_i^- - v_{i+1}^-)>0.
\end{equation}

These formulas can be found, for example, in Sention 4 of \cite{W90a}.

Concerning a floor collision $(0, 1)$ at time $t_k$, the
transformations are

\begin{equation}\label{floorcollision}
\begin{aligned}
  & \delta h_1^+(t_k)=\delta h_1^-(t_k) \\
  & \delta v_1^+(t_k)=\delta v_1^-(t_k) + \frac{2\delta h_1^-(t_k)}{m_1 v_1^+(t_k)},
\end{aligned}
\end{equation}

see, for instance, Section 4 of \cite{W90a} or \cite{W98}.

In this section we will be studying \emph{non-singular} trajectory segments

\begin{equation*}
  S^{[0, T]}x_0 = \left\{x_t=S^t x_0 \big|\; 0\le T\le T\right\} 
\end{equation*}

of the flow $\{S^t\}$ with the symbolic collision sequence
$\Sigma=(\sigma_1, \sigma_2,\dots ,\sigma_N)$, where $\sigma_k=(i_k,
i_{k+1})$, $0\le i_k \le n-1$, $k=1,2,\dots,N$. The collision graph
$\mathcal{G}=\mathcal{G}(\Sigma)$ of $\Sigma$ has the set $\{0, 1,
\dots, n\}$ as its vertex set, and the unoriented edges of
$\mathcal{G}$ are the unordered pairs $\{i_k, i_{k+1}\}$,
$k=1,2,\dots, N$, counted without multiplicity.

According to Wojtkowski's arguments between the Theorem and
Proposition I of Section 5 of \cite{W90a}, in order to prove the
strict invariance of the cone field $C_1=\{Q_1 \ge 0\}$ along the
considered trajectory segment

\begin{equation*}
  S^{[0, T]}x_0 = \left\{x_t=S^t x_0 \big|\; 0\le T\le T\right\},
\end{equation*}
it is enough to prove that

\begin{enumerate}
  \item[$(A)$] for every non-zero vector $(0, \delta v)\in \mathcal{T}_{x_0}$ it is true that
  $Q_1(DS^T(0, \delta v))>0$, and
  \item[$(B)$] for every non-zero vector $(\delta h, 0)\in \mathcal{T}_{x_0}$ it is true that
  $Q_1(DS^T(\delta h, 0))>0$.
\end{enumerate}

Moreover, Wojtkowski's mentioned arguments in Section 5 of \cite{W90a}
actually prove property (A) above in the case when the collision graph
$\mathcal{G}(\Sigma)$ restricted to the vertex set $\{1.2.\dots.n\}$
is connected, i. e. all collisions $(i, i+1)$ with $i>0$ occur. Here
we briefly rephrase his ideas:

Formula \ref{floorcollision} shows that a tangent vector of the form
$(0, \delta v)$ is unchanged at any floor collision. Suppose now that
a collision $(i, i+1)$ ($1\le i\le n$) occurs at time $t_k$ and
$\delta h^-(t_k)=0$. Equation \ref{nonfloorcollision} shows that,
after pushing the tangent vector $(0, \delta v^-(t_k))$ through the
collision $(i, i+1)$, either

\begin{equation}
  Q_1(\delta h^+(t_k), \delta v^+(t_k))=\alpha_i(\delta v_i^-(t_k)-\delta v_{i+1}^-(t_k))^2 > 0,
\end{equation}
or
\begin{equation}
  \delta v_i^-(t_k)=\delta v_{i+1}^-(t_k)=\delta v_i^+(t_k)=\delta v_{i+1}^+(t_k).
\end{equation}

Thus, as long as $Q_1(DS^T(0, \delta v(0)))=0$ and the restriction of
the collision graph $\mathcal{G}$ on the vertex set $\{1,2,\dots,n\}$
is connected, we have that

\begin{equation}
  \delta v_i(t)=\delta v_1(0)
\end{equation}

for all $t$, $0\le t\le T$ and all $i=1,2,\dots,n$. Therefore, $\delta
v_i(t)=0$ for all $t$, $0\le t\le T$ and all $i=1,2,\dots,n$, thanks
to the usual reduction equation $\sum_{i=1}^n \delta v_i=0$.  This
finishes the proof of Property (A) in the case when the collision
graph $\mathcal{G}(\Sigma)$ contains all edges $(i, i+1)$ with $1\le
i\le n-1$.

\begin{remark}
  From now on, we will always assume that the collision graph $\mathcal{G}$ contains all collisions
  $(i, i+1)$ with $i>0$.
\end{remark}

Checking Property (B) for the non-singular trajectory segment $S^{[0,
    T]}x_0$ is, however, a bit harder and it requires a bit more on
the side of the collision graph $\mathcal{G}(\Sigma)$ than simply
containing all edges $(i, i+1)$, $i=1,2,\dots, n-1$.

First of all, we define the linear space of all \emph{neutral vectors} as follows:

\begin{equation}\label{neutral}
  \mathcal{N}_{x_0}^T = \left\{(\delta h(0), \delta
  v(0))\in\mathcal{T}_{x_0}\big|\; Q_1(\delta h(t), \delta v(t))=0,
  \text{ for } 0\le t\le T\right\}.
\end{equation}

Recall that $(\delta h(t), \delta v(t))=DS^t\left((\delta h(0), \delta v(0))\right)$. 

It floows from the time-evolution equations \ref{evolution},
\ref{nonfloorcollision}, and \ref{floorcollision} that

\begin{equation}
  DS^t\left((\delta h(0), \delta v(0))\right)=(\delta h(t), 0)
\end{equation}

for $0\le t\le T$, i.e. $\delta v(t)=0$. Indeed, at any floor
collision the form $Q_1$ increases by the amount of

\begin{equation}
  \frac{2(\delta h_1^-(t_k))^2}{m_1v_1^+(t_k)},
\end{equation}

which forces $\delta h_1^-(t_k)=\delta h_1^+(t_k)=0$ and, as a
corollary, $\delta v_1^-(t_k)=\delta v_1^+(t_k)$.  According to
\ref{nonfloorcollision}, at an $(i, i+1)$ collision ($i\ge 1$),
occuring at time $t_k$, the form $Q_1$ increases by the amount

\begin{equation}
  \alpha_i\left(\delta v_i^-(t_k)-\delta v_{i+1}^-(t_k)\right)^2,
\end{equation}

and this forces $\delta v_i^-(t_k)=\delta v_{i+1}^-(t_k)=\delta
v_i^+(t_k)=\delta v_{i+1}^+(t_k)$, thus $\delta v(t)=0$ for $0\le t\le
T$, as claimed.

So, the neutral vector takes the form $DS^t(\delta h(0), 0)=(\delta
h(t), 0)$ for $0\le t\le T$, and $\delta h(t)$ only changes at
ball-to-ball collisions $(i, i+1)$ (at time $t_k$) according to the
law

\begin{equation}\label{Ristar}
  \delta h^+(t_k) = R_i^*\delta h^-(t_k),
\end{equation}

whereas at a floor collision, occuring at time $t_k$, the law

\begin{equation}\label{identicaltozero}
  \delta h^-(t_k)=\delta h^+(t_k)=0
\end{equation}

is forced by the neutrality requirement $Q_1(\delta h(t), \delta
v(t))=0$. Accordingly, we define $R_0=I=R_0^*$, the $n\times n$
identity matrix.

Out of the considered collisions $\sigma_1,\sigma_2,\dots,\sigma_N$,
let the floor collisions be

$$
\sigma_{k_1}, \dots,\sigma_{k_m},
$$
$1\le k_1<\dots <k_m\le N$. The
above argument shows that the neutral space $\mathcal{N}_{x_0}^T$ can
be defined by the following system of homogeneous linear equations:

\begin{equation}\label{homogeneous}
  \mathcal{N}_{x_0}^T=\mathcal{N}(\Sigma, \vec{m}) \\
  =\left\{(\delta h(0), 0)\in \mathcal{T}_{x_0}\big|\; \Pi_1 R^*_{i_{k_l}}\dots R^*_{i_2}R^*_{i_1}\delta h(0)=0, \\
  \text{ for } l=1,\dots,m\right\},
\end{equation}
where $\Pi_1(\delta h_1,\delta h_2,\dots,\delta h_n)=\delta h_1$ is
the projection onto the first component.

We observe that the neutral space $\mathcal{N}_{x_0}^T =
\mathcal{N}(\Sigma, \vec{m})$ only depends on the n-tuple os masses
$m_1>m_2>\dots>m_n>0$ and on the symbolic collision sequence
$\Sigma=(\sigma_1,\sigma_2,\dots,\sigma_N)$.

\begin{definition}\label{neutrality}
  The non-singular trajectory segment $S^{[0, T]}x_0$ or, equivalently, the corresponding pair $(\Sigma, \vec{m})$ is said to be \emph{sufficient} iff
  \begin{equation*}
    \mathcal{N}_{x_0}^T = \mathcal{N}(\Sigma, \vec{m}) = \{0\}.
    \end{equation*}
Otherwise, these objects are said to be insufficient.
\end{definition}

\begin{remark}
  A concequence of the previous definition is that the sufficiency of
  $(\Sigma, \vec{m})$ implies that $\Sigma$ must contain all $n$ types
  of collisions.  Indeed, a missing $(0, 1)$ collision would mean that
  the system of homogeneous linear equations \ref{homogeneous}
  contains no equations, whereas, a missing collision $(i_0, i_0+1)$
  ($i_0=1,2,\dots,n-1$) would establish no connection between the
  subsystems of particles $\{1,2,\dots,i_0\}$ and $\{i_0+1,
  i_0+2,\dots,n\}$, thus preventing sufficiency.
\end{remark}

Summarizing what we have seen so far, we have

\begin{prp}\label{strictconeinvariance}
  The sufficiency of the trajectory segment $S^{[0, T]}x_0$ (or,
  equivalently, the sufficiency of $(\Sigma, \vec{m})$) is equivalent
  to the \emph{strict invariance} of the cone field
  \begin{equation*}
    C_1(x_t)=\left\{(\delta h, \delta v)\in\mathcal{T}_{x_t}\big|\; \sum_{i=1}^n \delta h_i\delta v_i\ge 0\right\}
  \end{equation*}
  along
  \begin{equation*}
     S^{[0, T]}x_0=\left\{x_t=S^t x_0\big|\; 0\le t\le T\right\},
  \end{equation*}
     that is, this sufficiency exactly means that
  \begin{equation*}
  DS^T(C_1(x_0))\subset\text{int}\left(C_1(x_T)\right).
  \end{equation*}
\end{prp}

According to the time evolution equations \ref{evolution},
\ref{nonfloorcollision}, \ref{floorcollision}, in order for the
derivative $DS^T$ of the flow to be defined on the tangent space
$\mathcal{T}_{x_0}$, it is not enough to know the pair $(\Sigma,
\vec{m})$, but one needs to know the relative velocities

\begin{equation}\label{relvel1}
\rho_k=\rho_k(\sigma_k)=v^-_{i_k}(t_k)-v^-_{i_k+1}(t_k)>0
\end{equation}

for all ball-to-ball collisions $\sigma_k=(i_k, i_k+1)$ (they play a
role in \ref{nonfloorcollision} as a part of $\alpha_{i_k}$) and the
velocities

\begin{equation}\label{relvel2}
\rho_k=\rho_k(\sigma_k)=v_1^+(t_k)
\end{equation}

for any floor collision $\sigma_k=(0, 1)$ that play a role in
\ref{floorcollision}, $k=1,2,\dots,N$.

Therefore, the \emph{strict cone invariance} formulated in
\ref{strictconeinvariance} is a property of the \emph{extended
symbolic sequence}

\begin{equation}
  \tilde{\Sigma}=\left((\sigma_1, \rho_1), (\sigma_2, \rho_2), \dots (\sigma_N, \rho_N)\right)
\end{equation}

and the vector of the masses $\vec{m}$.

The characterization of sufficiency with the strict cone invariance
\ref{strictconeinvariance} has the big advantage that it appears to be
\emph{combinatorially monotone}, i.e. if one inserts an additional
collision $(\sigma^*, \rho^*)$ into the extended symbolic sequence

\begin{equation*}
  \tilde{\Sigma}=\left((\sigma_1, \rho_1), (\sigma_2, \rho_2), \dots, (\sigma_N, \rho_N)\right)
\end{equation*}

between $(\sigma_k, \rho_k)$ and $(\sigma_{k+1}, \rho_{k+1})$, then
the sufficiency will not be lost by this insertion:

\begin{prp}[Combinatorial Monotonicity Property, CMP]\label{CMP}
  Assume that, for a given mass distribution $m_1>m_2>\dots >m_n$, the extended symbolic sequence
\begin{equation*}
  \tilde{\Sigma}=\left((\sigma_1, \rho_1), (\sigma_2, \rho_2), \dots, (\sigma_N, \rho_N)\right)
\end{equation*}

is sufficient. Let the extended symbolic sequence

\begin{equation*}
  \tilde{\Sigma}^*=\left((\sigma_1, \rho_1), \dots ,(\sigma_k, \rho_k), (\sigma^*, \rho^*), (\sigma_{k+1}, \rho_{k+1}),
  \dots ,(\sigma_N, \rho_N)\right)
\end{equation*}

be obtained from $\tilde{\Sigma}$ by the insertion of $(\sigma^*, \rho^*)$, as indicated above. Then $\tilde{\Sigma}^*$ also satifies the
strict cone invariance property formulated in \ref{strictconeinvariance}.
\end{prp}

\begin{proof}

For any number $l$, $l=1,2,\dots,N$, let $D_l$ be the derivative of
the flow (in terms of $(\delta h, \delta v)$, as always) resulting
from the collision $(\sigma_l, \rho_l)$, and $D^*$ be the derivative
of the flow resulting from the collision $(\sigma^*, \rho^*)$.  We
have

\begin{equation*}
\begin{aligned}
  & D_N D_{N-1}\dots D_1(C_1(x_0))\subset\rm{int}(C_1(x_T)), \\
  & D^*D_k D_{k-1}\dots D_1(C_1(x_0))\subset D_k D_{k-1}\dots D_1(C_1(x_0)),
\end{aligned}
\end{equation*}

thus

\begin{equation*}
  D_N\dots D_{k+1}D^*D_k\dots D_1(C_1(x_0))\subset
  D_N D_{N-1}\dots D_1(C_1(x_0))\subset\rm{int}(C_1(x_T)).
\end{equation*}
\end{proof}

Since the coefficients of the system of homogeneous linear equations
\ref{homogeneous} are given rational functions of the masses $m_1,
m_2,\dots ,m_n$, we immediately obtain

\begin{prp}[Dichotomy]\label{dichotomy}
  For any given symbolic sequence
  $$
  \Sigma=(\sigma_1, \dots ,\sigma_N)
  $$
either

  \begin{enumerate}
  \item[$(D_1)$] for almost every n-tuple of masses $\vec{m}$
    ($m_1>m_2>\dots>m_n>0$, the exceptional set being a proper algebraic subset
    of $\mathbb{R}^n$) it is true that $\mathcal{N}(\Sigma, \vec{m})=\{0\}$, or
  \item[$(D_2)$] for every mass vector $\vec{m}$ we have $\mathcal{N}(\Sigma, \vec{m})\ne \{0\}$.
  \end{enumerate}
\end{prp}

\begin{definition}
  In the case $(D_1)$ above the symbolic sequence $\Sigma$ itself is
  said to be sufficient, otherwise it is said to be insufficient.
\end{definition}

Thanks to the characterization result \ref{strictconeinvariance} and
the Combinatorial Monotonicity Property \ref{CMP} we get

\begin{cor}\label{msigmacmp}
  For any given mass distribution $\vec{m}$, for any symbolic
  collision sequence $\Sigma$, and for any (not necessarily
  contiguous) subsequence $\Sigma_1$ of $\Sigma$, the sufficiency of
  the pair $(\Sigma_1, \vec{m})$ implies the sufficiency of $(\Sigma,
  \vec{m})$.
\end{cor}

Thanks to \ref{dichotomy}, we obtain the typical, i.e. the mass
distribution-free version of \ref{msigmacmp}:

\begin{cor}\label{sigmacmp}
  For any symbolic collision sequence $\Sigma$ and for any (not necessarily contiguous) subsequence $\Sigma_1$ of $\Sigma$,
  the sufficiency of $\Sigma_1$ implies the sufficiency of $\Sigma$.
\end{cor}

Next, we take a quick look at the limiting case $m_1=m_2=\dots
=m_n$. In this case, when a ball-to-ball collision $(i, i+1)$ occurs,
we can simply swap the labels $i$ and $i+1$ of the involved
particles. In this way, by performing all these dynamic label changes,
the entire flow becomes the independent motion of the $n$ particles,
each of them falling at unit acceleration and bouncing back
elastically from the floor. (Doing so, they freely pass through each
other.)

Observe that, in the case of equal masses, the transformation matrices
$R_i^*$ in \ref{Ristar} are the reflections

\begin{equation*}
  R_i=R_i^*=
  \left[
    \begin{array}{cc}
      0 & 1 \\
      1 & 0
    \end{array}
  \right],
\end{equation*}

thus the components $\delta h_i(t)$ will be independent of the time
$t$ after the dynamic re-labeling.

The above makes it clear that the trajectory segment $S^{[0, T]}x_0$
is sufficient in the sense of \ref{neutrality}, provided that all
particles hit the floor.  (After the dynamic re-labeling of the
particles, of course.) This means that, for the limiting system
$m_1=m_2=\dots =m_n$ there exists a uniform time treshold $\tau_0>0$
such that $S^{[0, T]}x_0$ is sufficient, provided that $T\ge
\tau_0$. By continuity, this property extends to every falling ball
system $m_1>m_2>\dots>m_n>0$ with $m_n/m_1\ge 1-\epsilon_0$ for some
fixed $\epsilon_0>0$. However, this property, along with the CMP
\ref{sigmacmp}, implies that there is a large enough
$K\in\mathbb{Z}_+$ such that every symbolic sequence
$\Sigma=(\sigma_1, \sigma_2,\dots,\sigma_N)$ is sufficient if $\Sigma$
contains at least $K$ consecutive, contiguous, connected
subsequences. (Recall that a subsequence is connected when its
collision graph is connected or, equivalently, the subsequence
contains all types of collisions.) However, the system of falling
balls obviously possesses the property that every trajectory $S^{[0,
    \infty )}x_0$ contains infinitely many appearances of each
  collision type. Thus we obtain the main result of this section:

\begin{theorem}[Main Result]\label{mainresult}
  For almost every n-tuple of masses $m_1>m_2>\dots>m_1>0$ (the
  exceptional set being a countable union of proper algebraic subsets
  of $\mathbb{R}^n$) the following is true: Every non-singular
  positive orbit $S^{[0, \infty )}x_0$ is sufficient, hence possesses
    the property of eventually strict cone invariance.
\end{theorem}

\section{Conclusions}

Here we summarize the corollaries of our Theorem \ref{mainresult}. These
corollaries will be \emph{conditional} on the condition that the
following property is possessed by the falling ball system:

\begin{clm}[Transversality Condition]\label{transversality}
  Singularities of different order are transversal to each other.
  Analogously, stable and unstable invariant manifolds are transversal
  to all singularities.
\end{clm}

Our goal is to prove the Ergodic Conjecture of Wojtkowski \cite{W90a}
for almost every falling ball system, assuming the above
transversality condition. The exceptional systems are the ones for
which our main result \ref{mainresult} is false, so they form a
countable union a proper algebraic sets in terms of the masses
$\vec{m}$. For this purpose we are going to use the celebrated Local
Ergodicity Theorem (LET, for short) of \cite{L-W95}, so we investigate
the first return map $T:\; U_0\to U_0$ of the billiard map to a
suitably small, open neighborhood of a hyperbolic phase point $x_0$
with at most one singularity on its trajectory.

First of all, the second part of \ref{transversality} is the so called
Chernov-Sinai Ansatz, a condition of the LET of \cite{L-W95}.

Second, the transversality of singularites guarantees that the set of
double singularities is a countable union of smooth, codimension-two
($\ge 2$) submanifolds, hence a slim set, negligible in dynamical
considerations, see Section 2 of \cite{K-S-Sz92}, esp. Lemma 2.12
there. Thus, we may safely assume that the considered hyperbolic phase
point $x_0$ has at most one singularity on its
trajectory. Hyperbolicity means, of course, that the cone field $C_1$
is strictly invariant along the trajectory of $x_0$.

Third, the second part of \ref{transversality} is Property 5' of
\cite{Ch93}, thus the result of \cite{Ch93} guarantees the Proper
Alignment of Singularities, another condition of the LET in
\cite{L-W95}.

Since the first return map $T:\; U_0\to U_0$ enjoys strict $C_1$-cone
invariance, according to Proposition 6.1 of \cite{L-W95}, the minimum
$Q_1$-expansion rate $\sigma$ of $T$ is uniformly bigger than
$1$. This establishes the Non-Contraction Property and the Strict
Unboundedness Property for the first returm map $T:\; U_0\to U_0$, two
more conditions of the LET.

Fourth, the Regularity of the Singular Set condition of \cite{L-W95}
is obtained as follows:

\begin{enumerate}
  
\item[$(1)$] The phase points with two singularities on their
  trajectory form a \emph{slim} set, as stated before.

\item[$(2)$] The smooth components of different singularities do not
  accumulate at any point of the phase space, since there are no
  infinitely many collisions in finite time (Proposition A.1 of the
  Appendix of \cite{S96}), and the horizon is finite.

\item[$(3)$] The derivatives of the singularities do not ``explode'',
  since there are no grazing (tangential) singularities in the falling
  ball system, only of corner type.

\end{enumerate}

Finally, the Local Ergodicity Theorem of \cite{L-W95} states that the
entire open neighborhood of the hyperbolic phase point $x_0$ (with at
most one singularity on its trajectory) belongs to a single ergodic
component of the falling ball system. For almost every mass vector
$\vec{m}$, the set of phase points $x_0$ not having the two properties
required above (strict cone invariance and at most one singularity
along its trajectory) is a slim set. According to Lemma 2.12 of
\cite{K-S-Sz92}, such slim sets are unable to cut the phase space into
different, open ergodic components.  Thus, we conclude:

\begin{theorem}[Main Corollary]\label{maincorollary}

  Assuming the condition \ref{transversality}, for almost every mass
  vector $m_1>m_2>\dots >m_n>0$ the falling ball system is ergodic.

\end{theorem}

\bibliographystyle{amsplain}

\end{document}